\documentclass{aic}

\usepackage{amssymb}
\usepackage{amsfonts}
\usepackage{amsthm}
\usepackage{amsmath}
\usepackage{amscd}
\usepackage[T1]{fontenc}
\usepackage{t1enc}
\usepackage{indentfirst}
\usepackage{graphicx}
\usepackage{graphics}
\usepackage{pict2e}
\usepackage{epic}
\usepackage{tikz}

\newtheorem{theo}{Theorem}[section]
\newtheorem{definition}{Definition}[section]
\newtheorem{prop}[theo]{Proposition}
\newtheorem{lemma}[theo]{Lemma}
\newtheorem{coro}[theo]{Corollary}

\newcommand{\ZZn}{{\mathbb{Z}_{2}^{n}}}
\newcommand{\ZZd}{{\mathbb{Z}_{2}^{d}}}
\aicAUTHORdetails{%
  title = {Unitary Signings and Induced Subgraphs 
of Cayley Graphs of $\ZZn$}, 
  author = {Noga Alon and Kai Zheng},
  plaintextauthor = {Noga Alon, Kai Zheng},
  plaintexttitle = {Unitary signings and induced subgraphs of Cayley graphs of Z_2^n}, 
}   

\aicEDITORdetails{%
   year={2020},
   number={11},
   received={26 March 2020},   
   published={9 November 2020},  
   doi={10.19086/aic.17912},      
}   

\begin{document}

\begin{frontmatter}[classification=text]

\title{Unitary Signings and Induced Subgraphs \\
of Cayley Graphs of $\ZZn$} 

\author[nalon]{Noga Alon\thanks{Supported in part by NSF grant DMS-1855464, ISF grant 281/17, BSF grant 2018267
and the Simons Foundation.}}
\author[kzheng]{Kai Zheng}
\begin{abstract}
Let $G$ be the Cayley graph of the elementary abelian $2$-group
$\mathbb{Z}_2^{n}$ with respect to a set $S$ of size $d$. We
prove that
for any such $G, S$ and $d$,
the maximum degree of any induced subgraph of $G$ on any set of more
than half the
vertices is at least $\sqrt d$. This is deduced from the recent breakthrough
result of Huang who proved the above for the
$n$-hypercube $Q^n$, in which the set of generators
$S$ is the set of all vectors of Hamming weight $1$. 
Motivated by his method we define and study unitary signings of adjacency 
matrices of graphs, and compare them to the orthogonal signings of Huang.
As a byproduct, we answer a
recent question of Belardo, Cioab\u{a}, 
Koolen, and Wang about the
spectrum of signed $5$-regular graphs.
\end{abstract}
\end{frontmatter}

\section{Introduction} \label{sec: intro}
For a group $B$ and a set $S \subseteq B$ satisfying
$S=S^{-1}$, the Cayley graph of  
$B$ with a generating set $S$, denoted $\Gamma(B, S)$, is  
the graph whose set of vertices 
is the set of elements of $B$, in which the set of edges is the 
set of all (unordered) pairs $\{g, gs\}$ where $g \in B$ and $s \in S$.
Any Cayley graph is regular and vertex transitive. In this note we  
consider Cayley graphs of the elementary abelian $2$-group 
$B = \mathbb{Z}_{2}^{n}$. Here we use additive notation and note that
$S=-S$ for every subset of the group. A particular example that 
has been the subject of recent attention is the $n$-dimensional hypercube, 
$Q^{n}$. This is the Cayley graph of $B = \mathbb{Z}_{2}^{n}$ whose 
generating set $S$ consists of all length $n$ boolean vectors 
with Hamming weight $1$. A recent breakthrough 
result of Huang \cite{Hu} asserts that any 
induced subgraph $H$ of $Q^{n}$ on a set of $2^{n-1}+1$ vertices
has maximum degree $\Delta(H)$ at least $\sqrt{n}$.
This, together with a result of Gotsman and Linial \cite{GL}, 
confirmed the Sensitivity Conjecture of Nisan and Szegedy \cite{NS}.

Huang's result also improves a previous lower bound of
$(\frac{1}{2} - o(1))\log_2{n}$ of Chung, F{\"u}redi, Graham, and
Seymour \cite{CFGS}, while another result of \cite{CFGS} shows that
his $\sqrt n$ bound is tight.

Our first result extends the lower bound to any Cayley graph
of $\ZZn$. 
These graphs are sometimes called \emph{cubelike graphs}, see, for example,
\cite{MPRR}.
\begin{theo}
\label{t11}
For any Cayley graph $G=\Gamma(\mathbb{Z}_{2}^{n}, S)$
of $\mathbb{Z}_{2}^{n}$ with respect to any
generating set $S$, and for
any subset $U \subseteq \mathbb{Z}_{2}^{n}$
of cardinality $|U| > 2^{n-1}$, the maximum degree of the induced subgraph
$H$ of $G$ on $U$ satisfies $\Delta(H) \geq \sqrt{|S|}.$
\end{theo}
The proof, presented in the next section, is a simple application of
Huang's result. 

To prove his result, Huang shows that $Q^n$ admits
an {\em orthogonal signing}. This is a matrix obtained from the adjacency
matrix of $Q^n$ by replacing every nonzero entry by $1$ or $-1$ so that
the resulting matrix, call it $A$,
is symmetric and has orthogonal rows. Therefore, each
eigenvalue of this matrix is either $+\sqrt n$ or $-\sqrt n$, and as the
trace is zero, exactly half the eigenvalues are $+\sqrt n$ and half are
$-\sqrt n$. By Cauchy's Interlace Theorem (see, e.g., \cite{Hw} for 
a short proof) this implies that the largest eigenvalue of any principal
minor of size at least $2^{n-1}+1$ of $A$ is $\sqrt n$. The result
about the maximum degree of any induced subgraph of $Q^n$ on a set
$U$ of more than half the vertices follows from the simple fact that 
any eigenvalue of a matrix in which all nonzero entries have absolute
value $1$ is at most the size of the largest support of a row  of the
matrix. 

Motivated by this proof we define a {\em unitary signing}
\footnote{While similar terminology is used, orthogonal and 
unitary signings are not in general orthogonal or unitary matrices, 
which are matrices $M$ such that $MM^t$ or $MM^*$ is the identity.}
of a graph. This is a Hermitian
matrix obtained from the adjacency 
matrix of the graph by replacing each nonzero entry
by a complex number of norm $1$, 
so that the (complex) inner product of any two distinct
rows is $0$.  Huang's argument works in the complex case
as well, implying that if a $d$-regular graph admits a unitary signing,
then the maximum degree of any induced subgraph of it on a set of more
than half the vertices is at least $\sqrt d$. 
Unitary signings of graphs appear to be interesting in their own right. 

In Section \ref{sec: unitary signings} we prove the following result, providing additional 
examples of Cayley graphs of $\mathbb{Z}_{2}^{n}$ with 
orthogonal and unitary signings.

\begin{theo}
\label{t21}
Let $T = \{e_1, \ldots, e_n \}$ and let $U = \{E_2, \ldots, E_n \}$
where $e_i$ is the length $n$ boolean vector of Hamming weight $1$
with $1$ in its
$i$th coordinate and $E_i = \sum_{k=1}^{k = i} e_k$.
Let $G$ be a Cayley graph of $\ZZn$ with a generating set $S \subseteq
T \cup U$ so that $|S \cap U| \leq 1$. Then $G$ admits 
a unitary signing.
\end{theo}

Regarding the necessity of using complex numbers in unitary signings, 
we show in Section \ref{sec: un vs or} that some graphs have unitary signings
but do not have orthogonal signings. Let $Q_{+}^n$ denote the Cayley graph
of $\ZZn$ with the generating set $S=\{e_1, \ldots ,e_n,E_n\}$. 
\begin{prop}
\label{Pr: No Real Signing}
For $n \geq 2$, $Q_{+}^{n}$ has an orthogonal signing if and only 
if $n \equiv 0,3 \pmod{4}$. In particular 
for $n \equiv 1, 2 \pmod{4}$, $Q_{+}^{n}$ has a unitary 
signing but no orthogonal signing.
\end{prop}

Finally, as a byproduct we answer a question of 
Belardo, Cioab\u{a}, Koolen, and Wang \cite{BCKW} 
about the spectrum of signed $5$-regular
graphs.

\section{The proof of Theorem \ref{t11}}
\vspace{0.2cm}

\noindent
{\bf Proof of Theorem \ref{t11}:}\,
Let $G$ be a Cayley graph of $\ZZn$ with generating set $S=\{s_1,s_2,
\ldots ,s_d\}$. We may and will 
assume without loss of generality that $G$ is connected, that is,
the vectors of
$S$ span the vector space $\ZZn$. 
Indeed, otherwise the graph $G$ consists of
isomorphic connected components on the cosets of $Span(S)$. Any set of
more than half the vertices of $G$ contains more than half the vertices
of at least one of these components, and we can thus deduce the result
from the one for the connected case.

Let $Q^d$ be the $d$-dimensional
hypercube. This is the Cayley graph of $\ZZd$ with respect to the
set of generators $\{e_1,e_2, \ldots ,e_d\}$ consisting of all the
vectors of Hamming weight $1$ in $\ZZd$. Let $T$ be the linear
transformation from $\ZZd$ to $\ZZn$ defined by $T(e_i)=s_i$
for all $1 \leq i \leq d$. Since $S$ spans $\ZZn$, $T$ is onto, and the
inverse image of any element of $\ZZn$ contains exactly
$2^{d-n}$ elements. Let $U \subseteq \ZZn$ be a set of more than 
half the vertices of $G$. Then the inverse image $W=T^{-1}(U)$
of $U$ contains more than half the elements of $\ZZd$. 
By Huang's result there are distinct $w, w_1,w_2, \ldots ,w_t
\in W$, with $t \geq \sqrt d$, such that $w$ is adjacent 
in $Q^d$ to each $w_i$, $1 \leq i \leq t$. Therefore, the vertex
$T(w) \in U$ is adjacent in $G$ to each of the $t \geq \sqrt d$ 
distinct vertices $T(w_1), T(w_2), \ldots ,T(w_t)$, which all lie in
$U$. This completes the proof. \hfill $\Box$

\section{Unitary signings of graphs} \label{sec: unitary signings}

As mentioned in the introduction, we consider here
unitary and orthogonal signings of graphs, defined as follows.
\begin{definition}
A {\em unitary signing}
of a graph is a Hermitian matrix obtained from the 
adjacency matrix of the graph by replacing each nonzero entry by
a complex number of norm $1$, so that the (complex) inner product of
any two distinct rows is $0$. If all nonzero entries used are real,
that is, all are $+1$ or $-1$, this is an {\em orthogonal signing}.
\end{definition}
It is shown in \cite{Hu} that the hypercube $Q^n$ admits an orthogonal
signing. We extend this result in Theorem \ref{t21} which is proved
in this section.
Note that any Cayley graph of $\ZZn$ with a generating set consisting
of $n+1$ elements which span $\ZZn$ is isomorphic to one of the
graphs described in this theorem.

A useful fact about unitary signings of Cayley graphs
of $\ZZn$ is
that we can view them as sums of what may be called 
{\em edge signings}, as stated in the next
simple lemma.
\begin{lemma}
\label{lm: add signings}
Let $G = \Gamma(\mathbb{Z}_{2}^{n}, S)$,
and suppose $S = \{a_1, \ldots, a_m \}$ where each $a_j$
is distinct. If $A_j$ is a unitary signing of
$\Gamma(\mathbb{Z}_{2}^{n}, \{a_j\})$ and the matrices
$A_j$ anticommute, then
$A = \sum_{j=1}^{j=m} A_j$ is a unitary signing of $G$.
\end{lemma}
\begin{proof}
Obvious from the definition.
\end{proof}

We refer to a signing of a single element Cayley graph, 
$\Gamma(\mathbb{Z}_{2}^{n}, \{e\})$, as simply a signing of 
the edge $e$ or an edge signing of $e$. If the signing is also a unitary signing, we may call it a unitary signing of $e$ or a unitary edge signing. Such a Cayley graph is a perfect matching on $\ZZn$.

A convenient way of constructing edge signings when $n \geq 2$ is to 
use Kronecker products of the following $2 \times 2$ Hermitian matrices.
The first one is the identity, and the other three are 
known as Pauli matrices, and generate the Clifford Algebra of $\mathbb{R}^3$. 
\begin{align*}
I_2 &=
\left[
\begin{array}{cc}
1 & 0 \\
0 & 1 
\end{array}
\right], \\
R_0 &=
\left[
\begin{array}{cc}
1 & 0 \\
0 & -1 
\end{array}
\right], \\
R_1 &=
\left[
\begin{array}{cc}
0 & 1 \\
1 & 0 
\end{array}
\right], \\
R_2 &=
\left[
\begin{array}{cc}
0 & i \\
-i & 0 
\end{array}
\right]. \\
\end{align*}

With this notation, it is clear that 
$A_n \otimes \cdots \otimes A_1$ is a unitary signing 
of $e = (a_n, \ldots, a_1)$ if $A_j \in \{I_2, R_0 \}$ whenever $a_j = 0$, 
and $A_j \in \{R_1, R_2 \}$ whenever $a_j = 1$. 
We will sometimes refer to $A_i$ as the matrix in the $i$th position 
of the Kronecker product. Moreover, we will also make liberal use of 
the fact that $\{ R_0, R_1, R_2 \}$ is an anticommuting family while 
all other pairs of the above four matrices commute. 
When combined with the mixed product property, this fact provides 
an easy way to check if two signings of the above form anticommute. 
Indeed, if $A = A_n \otimes \cdots \otimes A_1$ and 
$B = B_n \otimes \cdots \otimes B_1$, with all 
$A_i, B_i \in \{I_2, R_0, R_1, R_2 \}$, then $A$ and $B$ 
either anticommute or commute. They anticommute if and only if 
an odd number of the pairs $A_i, B_i$  for $1 \leq i \leq n$
anticommute. 
In order to prove Theorem \ref{t21} we
construct anticommuting families of unitary edge signings using this method.
\begin{lemma}
\label{lm: AC}
For any generating set $S \subseteq \mathbb{Z}_{2}^{n}$ 
as described in Theorem \ref{t21}, there exists an anticommuting 
family of unitary edge signings $\{A_e\}_{e \in S}$.
\end{lemma}
\begin{proof}
    Let $A_{e_i} = M_n \otimes \cdots \otimes M_1$ where 
\[
\begin{cases}
    M_j &= R_0,  \quad \; \; \; \text{if} \quad j > i\\
    M_j &= R_1,  \quad \; \; \; \text{if} \quad j = i\\
    M_j &= I_2,  \quad \; \; \; \; \text{if} \quad j < i.\\
\end{cases}
\]
It is not difficult to check 
that $\{A_{e_i}\}_{i \in [n]}$ is an anticommuting 
family of unitary edge signings where $[n]$ denotes the integers from $1$ to $n$. 
In fact, they are also orthogonal signings and $\sum_{i=1}^{i=n} A_{e_i}$ is the matrix that Huang 
constructs in Lemma 2.2 of \cite{Hu}. To finish the proof, 
it remains to show that for any $E_i \in U$, there is a unitary signing 
for the edge $E_i$ that anticommutes with each matrix 
in $\{A_{e_i}\}_{i \in [n]}$. 
First define $B_k = M_k \otimes \cdots \otimes M_1$ where  
\[
\begin{cases}
    M_j &= R_1,  \quad \; \; \; \text{if $k-j$ is odd} \\
    M_j &= R_2,  \quad \; \; \; \text{if $k-j$ is even.} \\
\end{cases}
\]
The key property of the expansion of $B_k$ is that each $R_1$ 
has an odd number of matrices to the left of it, while each 
$R_2$ has an even number of matrices to the left of it. 
$E_i$'s edge signing is as follows, 
\[
A_{E_i} = 
\underbrace{R_0 \otimes \cdots 
\otimes R_0}_{n - i \text{ times}} \otimes B_i.
\] 
To check that any $\{A_{E_i}\} \cup \{A_{e_i}\}_{i \in [n]}$ 
is an anticommuting family of size $n+1$ we show that $A_{E_i}$ 
and $A_{e_j}$ anticommute for any $i, j$. We count the 
number of anticommuting pairs of matrices in their Kronecker 
product expansions. If $i \leq j$ there is one 
anticommuting pair: the $R_1$ 
from $A_{e_j}$ with either an $R_0$ if $i<j$ or an $R_2$ if $i = j$. 

If $i > j$ then $A_{e_j}$ has $R_1$ in the $j$th position, 
while $A_{E_j}$ can have either $R_1$ or $R_2$. If it 
has $R_1$ as well, then the $j$th position is not an 
anticommuting pair, but we have an odd number of $R_2$ and $R_1$'s 
to the left of the $j$th position as mentioned above. 
These $R_2$ and $R_1$ matrices all anticommute with the 
corresponding $R_0$'s in $A_{e_i}$'s expansion. 
Likewise, if we have $R_2$ in the $j$th position, then all the 
$R_1$'s and $R_2$'s in $B_i$ to the left of and including the 
$j$th position form anticommuting pairs and again this is an odd 
number of pairs. In each case every other pair commutes as 
it is either $R_0R_0$ or a pair with one of the matrices being $I_2$.
\end{proof}

We can now prove Theorem \ref{t21}. Let $G$ be a 
Cayley graph of $\ZZn$ with generating 
set $S$ as described in the theorem.
By Lemma \ref{lm: AC}, there is a family of anticommuting matrices 
consisting of unitary signings for the edges in $S$. The sum of 
these matrices is a unitary signing of $G$ by Lemma \ref{lm: add signings}.
\hfill  $\Box$

It is worth noting that there are additional pairs of anticommuting 
matrices outside of the family described in Lemma \ref{lm: AC}. 
This leads to unitary signings of certain Cayley graphs 
of $\ZZn$ with generating sets consisting of vectors 
from $T \cup U$ and up to two vectors from $U$. 

\begin{coro}
Any Cayley graph of $\ZZn$ with generating set consisting of a 
subset of $T = \{e_1, \ldots, e_n\}$ and exactly two 
vectors $E_i, E_j \in U = \{E_2, \ldots, E_n \}$ 
such that $i > j$, $i$ is odd and $j$ is even, admits a unitary signing.
\end{coro}
\begin{proof}
Using the construction from 
Lemma \ref{lm: AC}, it is easy to check that if $i > j$, $i$ is odd 
and $j$ is even then $A_{E_i}$ and $A_{E_j}$ anticommute. 
Thus, the edge signings constructed in Lemma \ref{lm: AC} 
provide the desired unitary signing. 
\end{proof}

\section{Unitary versus orthogonal signings} \label{sec: un vs or}

One may wonder if the use of complex numbers is needed in the
proof of Theorem \ref{t21}.
We show that complex signings are indeed necessary in some 
cases.  Recall that 
$Q_{+}^{n} = \Gamma(\mathbb{Z}_{2}^{n}, S)$ is the Cayley graph
of $\mathbb{Z}_{2}^{n}$
with $S = \{e_1, \ldots, e_n, E_n \}$. This graph is simply 
the $n$-dimensional hypercube with one additional element being the
all $1$ vector added 
to its generating set. 
In this section we start with the proof of Proposition 
\ref{Pr: No Real Signing} stated in Section \ref{sec: intro}.

The reverse direction of this Proposition 
is immediate as for $n \equiv 0,3 \pmod{4}$ and $n \geq 2$, 
the unitary signing of $Q_+^n$ constructed in Theorem \ref{t21} 
only uses real numbers and is thus an orthogonal signing as well.

To complete the proof of Proposition \ref{Pr: No Real Signing} we first discuss a simple way to determine whether certain graphs have an orthogonal signing. The method works whenever the graph in question has the property that any two vertices with a common neighbor have exactly two common neighbors. In Cayley graphs of $\ZZn$ this holds if and only if the generating set of the graph is a Sidon set.
A Sidon set is a set with the property that any distinct pair of 
elements has a unique sum, that is $s_1+s_2 = s_3 + s_4$ if and only if $\{s_1, s_2\} 
= \{s_3, s_4\}$ for all $s_1, s_2, s_3, s_4 \in S$ satisfying $s_1 \neq s_2$ and $s_3 \neq s_4$. 
Note that the generating set of $Q^n$ is a Sidon set and the generating 
set of $Q_{+}^{n}$ is a Sidon set when $n \neq 3$. Thus, the 
graph $Q_+^n$ with $n \geq 2$ and $n \equiv 1, 2 \pmod{4}$ has the 
property that vertices with a common neighbor have exactly two common 
neighbors. It is worth noting that there are other 
well known strongly-regular or distance-regular graphs 
with this property, including the 
Shrikhande graph, the $2$-dimensional Hamming graph 
over $\mathbb{Z}_4$, the Gewirtz graph, and the Klein graph,
see \cite{BCN}. 

Given an orthogonal signing of the adjacency matrix $M$ of
a graph, one can define a labelling, $f$, of the edges of the graph 
using $1$'s and $0$'s as follows. Let $e = ij$ be an edge, then put

\[
    f(e)= 
\begin{cases}
    1,& \text{if } M_{ij} = -1\\
    0,              & \text{otherwise.}
\end{cases}
\]

For Cayley graphs $\Gamma(\ZZn, S)$ where $S$ is a Sidon set, 
we can then reformulate the problem of finding orthogonal signings as follows.

\begin{lemma}
\label{Lm: cycles}
Let $G = \Gamma(B, S)$ be a Cayley graph of $B=\mathbb{Z}_{2}^{n}$ 
where $S$ is a Sidon set. Let $C$ be the set  of all $4$-cycles 
in $G$ where the elements of $C$ are sets of $4$ edges that form a $4$ cycle. 
An orthogonal signing of $G$ exists if and only if 
there exists a labelling of the edges of the graph 
$f: E(G) \xrightarrow{} \{0, 1\}$, 
such that $\sum_{e \in c} f(e) = 1$ for all cycles $c \in C$, where
each sum here is computed in the group $\mathbb{Z}_2$.
\end{lemma}
\begin{proof}
Since $S$ is a Sidon set, $G$ has the property that any two vertices 
with a common neighbor share exactly two common neighbors. 
Let $M_{ij}$ be the entry of the edge $e = ij$. 
If an orthogonal signing exists, then for any $4$-cycle, in order for 
the rows indexed by opposite vertices of the $4$-cycle to be orthogonal, 
the $4$-cycle must have exactly $3$ edges whose entry is $1$ or exactly 
$3$ edges whose entry is $-1$. The desired labelling of $f$ is thus the
one described above.

On the other hand, if such a labelling $f$ exists, one can define the signed 
adjacency matrix $M$ by $M_{ij} = -1$ if the edge $ij$ is labelled $1$ 
and $M_{ij} = 1$ if the edge $ij$ is labelled $0$. 
Then the inner products of rows $i$ and $j$ is the sum, over all 
length $2$ walks between $i$ and $j$, of the product of the edge 
entries of each walk. There are either $2$ or $0$ length 
$2$ walks between any two vertices. The result 
then follows from the assumption that in each $4$-cycle there is
an odd number of edges labelled $1$.
\end{proof}

We can now prove the remaining direction in Proposition \ref{Pr: No Real Signing}.

\begin{proof}
To show that $Q_+^n$ has no orthogonal signing when $n \geq 2$ and $n \equiv 1, 2 \pmod{4}$, it suffices, by Lemma \ref{Lm: cycles}, to find an odd number of cycles such that each edge used in a cycle is used in an even number of cycles.
Indeed, let $C_1, \ldots, C_{2k+1}$ be such a family
cycles. If an orthogonal signing exists, we 
have $\sum_{e \in C_{i}} f(e) = 1$, for each $C_i$. Adding 
the $2k+1$ equations together yields 
$\sum_{e \in \bigcup_{i=1}^{i=2k+1}C_i} f(e) = 1$. 
However, since each edge is used an even number of times, the left 
hand side is $0$, yielding a contradiction. 

Thus, it remains to show that an odd collection of cycles as above 
exists. We describe such a collection arranged in a 
staircase shape. This collection in $Q_+^5$ is depicted in Figure 1 
and the general construction, described in what follows, is similar. 

Consider the integer lattice with all coordinates $(x, y)$ such 
that $x+y \leq n+1$ and $n\geq x, y \geq 0$. Notice that 
connecting adjacent lattice points gives a staircase grid 
composed of $\frac{n(n+1)}{2}$ unit squares. We can view this 
structure as a graph $L_n$ whose set of vertices is
$V(L_n)= \{(x,y): x+y \leq n+1, n\geq x,y\geq0 \}$ where 
$(x_1, y_1)$ is adjacent to $(x_2, y_2)$ if and only if the two lattice points are exactly one unit apart. It is easy to see that the $4$-cycles of $L_n$ 
correspond exactly to the unit squares of the staircase shape 
and that as a result $L_n$ has $\frac{n(n+1)}{2}$ $4$-cycles.
We identify each vertex $(x,y) \in V(L_n)$ with a vertex 
in $V(Q_+^n)$ such that the edges  of $L_n$ are preserved. 
Define $f: V(L_n) \xrightarrow{} \{0, 1\}^n$ to be this mapping. 
Let $S = \{e_1, \ldots, e_n, E_n \}$ be the generating vectors 
of $Q_+^n$ as defined previously. 
Write $f(x,y)$ instead of $f((x,y))$, for simplicity and use $(a,b)$ to denote the edge between $a$ and $b$ 
when $a$ and $b$ are vertices (in $L_n$ or in $f(L_n)$).

Define 
\begin{align*}
    f(0,0) &= (0, \ldots, 0), \\
    f(x, 0) &= e_1 + \ldots + e_x, \quad \quad \quad \text{\quad for $x \geq 1$} \\ 
    f(0, y) &= e_1 + \ldots + e_{n+1-y}, \quad \quad \text{for $y \geq 1$} \\
    f(x, y) &= f(x, 0) + f(0, y).
\end{align*}

It is clear that $f$ preserves the edges of $L_n$, as $f(x+1, y) - f(x, y),  f(x, y+1) - f(x, y)\in S$. Moreover if the vertices $\{(x, y), (x+1, y), (x, y+1), (x+1, y+1) \}$ form a $4$-cycle in $L_n$, then $x+1+y+1 \leq n+1$ and 
\begin{align*}
    f(x+1, y) - f(x, y) = e_{x+1} \neq e_{n+1-y} = f(x, y+1) - f(x, y).
\end{align*}

Thus, no $4$-cycle of $L_n$ is mapped to a degenerate $4$-cycle, so $f(L_n)$ yields a collection of $\frac{n(n+1)}{2}$ $4$-cycles in the graph $Q_+^n$.
For $n \equiv 1, 2 \pmod{4}$ this is an odd number, 
so it remains to show that this collection of cycles uses each 
edge an even number of times. We can pair the edges 
of the $4$-cycles of $L_n$ so that in each pair, the two edges 
have the same image under $f$. This implies that the images of 
the $\frac{n(n+1)}{2}$ $4$-cycles of $L_n$ under $f$ form an odd 
collection of $4$-cycles with each edge used an even number of times
in $Q_+^n$. Note, first, that any edge in the interior can be paired 
with itself since it is used in two cycles. Next, for the boundary 
edges note that $((0,0), (1,0))$ can be paired 
with $((0, n), (1, n))$ as 
\begin{align*}
    f(1, n) &= e_1 + e_1 = 0 = f(0,0), \\
    f(1,0) &= e_1 = f(0,n).
\end{align*}
Similarly, $((0,0), (0, 1))$ can be paired with $((n, 0), (n, 1))$ as 
\begin{align*}
    f(n,1) &= E_n + E_n = 0 = f(0, 0), \\
    f(n, 0) &=  E_n  = f(0, 1).
\end{align*}
We can also pair the edge $((x, 0), (x+1, 0))$ 
with $((0, n+1-x), (0, n-x))$ for $1 \leq x \leq n-1$ as 
\begin{align*}
    f(x,0) &= e_1 + \cdots + e_x = f(0, n+1-x), \\
    f(x+1, 0) &= e_1 + \cdots + e_{x+1} = f(0, n-x). \\
\end{align*}
Lastly it remains to pair the edges of the form $((x, y), (x, y+1))$ 
and $((x, y), (x+1, y))$ for $1 \leq x,y \leq n-1$ and 
$x+y+1 = n+1$. These are the edges on the boundary, but not on
one of the axes. Notice that $((x, y), (x+1, y))$ and 
$((x, y), (x, y+1))$ are mapped to the same edge if $x+y+1 = n+1$. 
Indeed both are mapped to the edge $(e_{x+1}, 0)$. 
This exhausts all edges on the boundary implying that under 
$f(L_n)$, each edge appears in an even number of cycles. 
Thus, for $n \equiv 1, 2 \pmod{4}$, we get a collection 
of $\frac{n(n+1)}{2} \equiv 1 \pmod{2}$ cycles where each edge 
is used an even number of times. 
By Lemma \ref{Lm: cycles}, this implies that there 
is no real orthogonal signing of $Q_+^n$ 
when $n \equiv 1,2 \pmod{4}$. However, by  
Theorem \ref{t21}, a complex unitary signing does exist. 
\end{proof}
\begin{figure}[!h]
\centering
\begin{tikzpicture}[darkstyle/.style={circle,draw,fill=gray!40,minimum size=20}]
  \foreach \y in {0}
    \foreach \x in {0,...,5} 
       {\pgfmathtruncatemacro{\label}{2^\x-1}
       \node [darkstyle]  (\x\y) at (1.5*\x,1.5*\y) {\label};} 
  \foreach \x in {0}
    \foreach \y in {1,...,5} 
       {\pgfmathtruncatemacro{\label}{2^\y-1}
       \node [darkstyle]  (\x\y) at (1.5*\x,9 - 1.5*\y) {\label};} 
  \foreach \x in {1, ..., 5}
    \foreach \y in {\x, ..., 5}
        {\pgfmathtruncatemacro{\label}{2^\y-2^\x}
           \node [darkstyle]  (\x\y) at (1.5*\x,9 - 1.5*\y) {\label};}

  \foreach \x [remember=\x as \lastx (initially 0)]in {1,...,5} 
    \draw (\lastx0) -- (\x0) (\x0) -- (\x5) (\lastx0) -- (\lastx5);
    
  \foreach \y [remember=\y as \lasty (initially 1)]in {2,...,5} 
    \draw (0\lasty) -- (0\y) (0\y) -- (1\y) (0\lasty) -- (1\lasty);
  
  \foreach \x [remember=\x as \lastx (initially 1)]in {2,...,5} 
    \foreach \y [remember=\y as \lasty (initially \lastx)] in {\x, ..., 5}
      \draw (\lastx\y) -- (\lastx\lasty);
  \foreach \y [remember=\y as \lasty (initially 1)]in {2,...,5} 
    \foreach \x [remember=\x as \lastx (initially 1)] in {1, ..., \y}
      \draw (\lastx\y) -- (\x\y);
\end{tikzpicture}
\caption{A depiction of $f(L_5)$, or $15$ cycles in $Q_+^5$ where each edge is used two times. The vertex labels correspond to their binary representations as length $5$ Boolean vectors, e.g. $15$ is the vertex $(0, 1, 1, 1, 1)$.}
\end{figure}
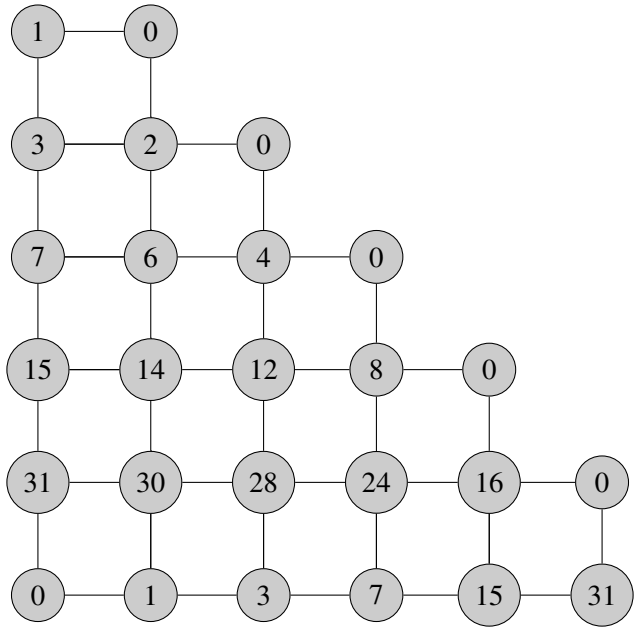

Notice that the construction above works whenever we have a 
cycle of length $2$ or $3$ modulo $4$ with distinct generators for a Cayley 
graph of $\mathbb{Z}_2^n$. We simply repeat the above, using such a
cycle of generators in place of $\{e_1, \ldots, e_n, E_n\}$. Thus, we get the following result on the existence of orthogonal signings 
for Cayley graphs of $\mathbb{Z}_2^n$.
\begin{prop}
Let $G$ be a Cayley graph of $\mathbb{Z}_2^n$ with generating 
set $S$. If there exist distinct $e_1, \ldots, e_k \in S$ 
such that $e_1 + \cdots + e_k = 0$, and 
$k \equiv 2,3 \pmod{4}$, then no orthogonal signing of $G$ exists.
\end{prop}
\begin{proof}
Immediate from the construction in the previous 
proof using the edges $e_1, \ldots, e_k$.
\end{proof}
It is interesting to note that our signing of $Q_+^4$, which, as can
be easily checked,
uses only real numbers, 
answers an open question posed recently by Belardo, Cioab\u{a}, 
Koolen, and Wang \cite{BCKW}. They studied real orthogonal signings of 
regular graphs and noted that any $4$-regular graph with an orthogonal 
signing, $A_s$, can be used to produce a $5$-regular 
graph with an orthogonal signing,
\[
B=
\left[
\begin{array}{cc}
A_{S} & I \\
I & -A_{S}
\end{array}
\right],
\] 
in a way similar to Huang's matrix construction. Specifically, 
the new $5$-regular graph is formed by taking two copies of the 
$4$-regular graph and adding a perfect matching between them joining
each vertex to its copy. By the work of McKee and Smyth 
\cite{MS} on $4$-regular graphs, this method yields many 
$5$-regular graphs with such an orthogonal signed adjacency matrix. 
Belardo, Cioab\u{a}, Koolen, and Wang ask if any other such examples exist, 
namely if there are $5$-regular graphs with orthogonal signings that
are not constructed by the above method. 
Our graph $Q_+^4$ is one such example. To show it is not one
of the examples described in \cite{BCKW} note
that all the graphs there have a cut forming  a perfect matching which
splits the vertex set into two equal parts. For $Q_+^4$ to be such a 
graph it would have to have a cut of size $8$. However, this is not the 
case as can be seen from the eigenvalues of its adjacency matrix. 
This follows from the following known fact, proved, for example, 
in the remark following Lemma 2.1 in \cite{AM}. 
For completeness we include the simple proof.

\begin{prop}
If $G = (V, E)$ is a regular graph with $|V| = n$ even, 
$U \subseteq V$, $|U| = \frac{n}{2}$, and $G$ 
has eigenvalues $\lambda_1 \geq \cdots \geq \lambda_n$, 
then $e(U, V - U) \geq \frac{d - \lambda_2}{4}n$, 
where $e(U, V-U)$ denotes the number of edges with endpoints 
in $U$ and $V - U$.
\end{prop}

\begin{proof}
    Let $A$ be the adjacency matrix 
of $G$ and let $f$ be the column vector indexed by $V$ defined by:
    \[
    f(u)= 
\begin{cases}
    1,& \text{if } u \in U \\
    -1,              & \text{if $u \notin U$,}
\end{cases}
\] 
where $f(u)$ is the coordinate of $f$ indexed by $u \in V$. We have,

\begin{align*}
    f^tAf = 2 \sum_{uv \in E}f(u)f(v) = 
2[-e(U, V-U) + \frac{nd}{2} - e(U, V-U)] = nd - 4e(U, V-U) \\
\end{align*}
and hence
\begin{align*}
    e(U, V-U) = \frac{nd}{4} - \frac{f^tAf}{4} \\
\end{align*}

One can bound $f^tAf$ using the second largest eigenvalue of $A$ 
since it is orthogonal to the all $1$ vector which is the eigenvector of
the largest eigenvalue. 
Thus, $f^tAf \leq \lambda_2||f||_2^2 =\lambda_2n$, and 
\begin{align*}
    e(U, V-U) \geq \frac{nd}{4} - \frac{\lambda_2n}{4} 
= \frac{d-\lambda_2}{4}n.
\end{align*} 
\end{proof}
It is well known (c.f., for example, \cite{Lo}, Problem 11.8)
that the eigenvalues of Cayley graphs of Abelian
groups can be expressed as character sums. This easily implies that
the second eigenvalue of $Q_{+}^{4}$ is $1$. Thus, by the previous 
proposition, any cut of $Q_{+}^{4}$ that splits the vertex set into 
two equal vertex classes
has at least $16$ edges. However, a perfect matching of $Q_+^4$ has 
exactly $8$ edges, so $Q_+^4$ cannot have the form described
in \cite{BCKW} in which there is such a cut.

\section{Concluding remarks}

By Theorem \ref{t11} the induced subgraph of any $d$-regular Cayley graph of $\ZZn$ on more than half the vertices has maximum 
degree at least $\sqrt d$. It is easy to see that this is not the case 
for Cayley graphs of other (abelian) groups. For example, 
if a group has an element of order $6$ then the Cayley graph in which the
only generators are such an element and its inverse is a vertex disjoint
union of $6$-cycles. This has a set of two thirds of the vertices so that
the induced subgraph on them is a matching, having maximum degree $1<
\sqrt 2$. 

The lower bound provided by Theorem \ref{t11} is  
tight for the $n$-cube,
as shown by the example described in 
\cite{CFGS}. It is therefore also tight for any Cayley graph obtained
from the $n$ cube by adding one additional generator, when $n$ is a perfect square, and is off by at most one for other $n$. For many Cayley
graphs of $\ZZn$, however, the bound is far from being tight. 
In particular for most Cayley graphs with at least $Cn$ generators,
for some absolute constant $C$, the lower bound can be
improved to linear in $d$. This follows from known
results about expanders and random Cayley graphs, as explained in
what follows.
An $(n,d,\lambda)$-graph is 
a $d$-regular graph on $n$ vertices in which the absolute value
of every eigenvalue besides the top one is at most $\lambda$.
The following result is proved in \cite{AC}. 
\begin{lemma}[\cite{AC}]
\label{AC}
Let $G = (V, E)$ be an $(n, d, \lambda)$ graph and let
$S$ be an arbitrary set of $\alpha n$ vertices of $G$. Then the average
degree $\overline{d}$ of the induced subgraph of $G$ on $S$ satisfies
\begin{align*}
    |\overline{d} - \alpha d| \leq \lambda (1-\alpha).
\end{align*}
\end{lemma}
The proof in \cite{AC} in fact  shows that it suffices 
to assume that the smallest (most negative) eigenvalue of
$G$ is at least $-\lambda$ to conclude that the average degree
is at least $\alpha d - \lambda (1-\alpha)$. (The assumption that
the second largest eigenvalue is at most $\lambda$ provides an
upper bound of $\alpha d + \lambda (1-\alpha)$ for this average degree).

For random Cayley graphs 
of $\ZZn$ with at least some $Cn$ generators,
$\lambda$ is smaller than, say, $d/2$ with high probability, as
proved in \cite{AR}. 
\begin{prop}[\cite{AR}]
\label{p41}
For every $1 > \delta > 0$ there exists a 
$C(\delta) > 0$ such that the following holds. Let $B$ be a 
group of order $n$, let $S$ be a set of $C(\delta)\log_2(n)$ 
random elements of $B$, where $S=S^{-1}$ and let 
$G$ be the Cayley graph of $B$ 
with respect to $S$. Then, with high probability (that is, with probability
that tends to $1$ as $n$ tends to infinity) the absolute value of
every eigenvalue of $G$ but the top one is at
most $(1-\delta)|S|$.
\end{prop}
Taking $\delta=1/2$ this shows that the lower bound in
Theorem \ref{t11} can be improved to $|S|/4$ for most 
Cayley graphs  of $\ZZn$ with at least $C(1/4)n$ generators.

Theorem \ref{t21} shows that any connected Cayley graph of 
$\ZZn$ with at most $n+1$ generators admits a unitary signing. 
It is not difficult to give additional examples of
Cayley graphs of this group with such a signing. In particular
the tensor product of any two graphs with a unitary signing
admits a unitary signing, as follows by the mixed product property.
Any $m$ by $m$ Hadamard matrix provides an orthogonal (and hence
unitary) signing of the complete graph with loops on $m$ vertices
in each vertex class. Tensor products of such a graph and graphs
covered by Theorem \ref{t21} are graphs admitting unitary
signings. On the other hand, there are lots of Cayley graphs of
$\ZZn$ that do not admit  such a signing. A large family of examples
are ones in which the generating set $S$ is a Sidon set with more
than  $2n+1$ elements. To prove this note that if
$S$ is a Sidon set, a unitary signing (squaring to $|S|I$) must
be the sum of $|S|$ anticommuting edge signings. This is because
if the unitary signing is 
decomposed into edge
signings, $e_1 + \cdots + e_{|S|}$, in the square of this sum
the nonzero entries of any product of two matrices  $e_i,e_j$
can intersect those of $e_p,e_q$ only if $\{i,j\}=\{p,q\}$. Thus
the only way for the sum to square to $|S|I$ is for
all of the edge signings to anticommute with each other.

It is known (c.f. \cite{Hr}) that the maximum possible number of complex
invertible anticommuting $N$ by $N$  matrices is
$2 \log_2 N+1$. In our case $N=2^n$ and thus we cannot have more than
$2n+1$ anticommuting matrices corresponding to edge signing (note that
each such matrix is invertible, as its square is the identity).
There are, however, much larger Sidon sets in $\ZZn$, 
of size $2^{n/2}$
(given by the columns of the parity check matrix of a BCH code with designed
distance $5$). Any Cayley graph of $\ZZn$ with a generating set which is a 
subset of more than $2n+1$ elements of such a large Sidon set cannot have
a unitary signing.

It will be interesting to understand better which (Cayley or non-Cayley) 
graphs admit unitary signings. It also seems interesting to investigate
the possible analogs of Theorem \ref{t11} for Cayley graphs of other groups.


\section*{Acknowledgments} 
We thank two anonymous reviewers for helpful suggestions
and comments.

\bibliographystyle{amsplain}


\begin{aicauthors}
\begin{authorinfo}[nalon]
  Noga Alon\\
  Department of Mathematics,\\
  Princeton University,\\
  Princeton, USA \\
  and\\
  Schools of Mathematics 
  and Computer Science,\\
  Tel Aviv University, \\
  Tel Aviv, Israel. \\
  nogaa\imageat{}tau\imagedot{}ac\imagedot{}il \\
\end{authorinfo}
\begin{authorinfo}[kzheng]
  Kai Zheng\\
  Department of Mathematics,\\
  Princeton University,\\
  Princeton, USA. \\
  kzheng\imageat{}princeton\imagedot{}edu \\
\end{authorinfo}
\end{aicauthors}

\end{document}